\documentclass[12pt,leqno]{amsart}
\usepackage{amsmath,amsthm}
\usepackage{amsfonts}

 \newtheorem{res}{Result}[section]
 \newtheorem{thm}[res]{Theorem}
 \newtheorem{example}[res]{Example}
 \newtheorem{remark}[res]{Remark}
 \newtheorem{prop}[res]{Proposition}
 \newtheorem{lem}[res]{Lemma}

\newtheorem{defn}[res]{Definition}

\numberwithin{equation}{section}
\def\tT{{\mathbf T}}
\def\d{d}
\def\defto{\buildrel def\over =}
\def\ft{f^{(2)}}
\def\mp{\mu^1_+}

\def\mpj{\mu^j_+}
\def\mmk{\mu^k_-}

\def\mm{\mu^1_-}
\def\T{{\mathcal T}}
\def\vp{{\varsigma_+}}
\def\P{{\mathcal P}}
\def\X{{\boldsymbol{X}}}
\def\tC{{\boldsymbol{C}}}
\def\C{\tC}
\def\bN{{\boldsymbol{N}}}
\def\tS{{\boldsymbol{S}}}
\def\Lto{\buildrel L^2\over \rightarrow}

\def\weak{\buildrel \hbox{ weak}^*\over \rightarrow}

\begin{document}
\bibliographystyle{plain}

\title[An RDE on the unit interval]{The noisy veto-voter model:
a Recursive Distributional Equation on [0,1]}

\date{}
\author{SAUL JACKA}
\address{Dept. of Statistics, University of Warwick, Coventry CV4 7AL, UK}
\email{s.d.jacka@warwick.ac.uk}
\author{MARCUS SHEEHAN}

\begin{abstract}
We study a particular example of a recursive
distributional equation (RDE) on the unit interval. We identify
all invariant distributions, the corresponding ``basins of
attraction" and address the issue of endogeny for the associated
tree-indexed problem, making use of an extension of a recent
result of Warren.
\end{abstract}

\thanks{The authors are grateful for many fruitful discussions with Jon
Warren on the topics of this paper.}

\maketitle
\centerline{{\today}}

\section{Introduction}
Let $M$ be a random variable, taking values in
$\mathbb{\overline{N}} = \{1, ... ;
\infty\}$, and let $\xi$ be an independent Bernoulli($p$) random variable.

We consider the following simple Recursive Distributional Equation (henceforth abbreviated as RDE):
\begin{equation}
\label{DE1} Y = \xi\prod_{i=1}^M Y_i+ (1-\xi)(1-\prod_{i=1}^M Y_i).
\end{equation}
Viewing (\ref{DE1}) as an RDE, we seek a stationary
distribution, $\nu$, such that if $Y_i$ are iid with distribution $\nu$
and are independent of $(M, \xi)$, then $Y$ also has distribution $\nu$.

We term (\ref{DE1}) the noisy veto-voter model since, if each $Y_i$ takes values
in $\{0,1\}$ with value $0$ being regarded as a 
veto, then the outcome is vetoed unless either (a) each voter i  \lq assents' 
($Y_i=1$ for each $1\leq i\leq M$) and
there's no noise ($\xi=1$) or (b) someone vetos, but is reversed by the
noise ($\xi=0$). 

The system was originally envisaged as modelling a
representative voting system applied to a veto issue. Thus each
representative votes according to their constituency if $\xi=1$ or
reverses the decision if $\xi=0$. An alternative interpretation is as a
model for a noisy distributed error-reporting system. Here a $0$
represents an error report from a sub-system. Thus there is an error in
the system if there is an error in any sub-sytem (hence the veto
structure). Noise can reverse the binary (on-off) report from any
sub-system.

In this paper, we look for solutions to the RDE (\ref{DE1}) taking values
in $[0,1]$.

As observed in Aldous and Bandhapadhyay \cite{Aldous}, and
as we shall explain in a little more detail in section 2, we may think of
(families of) solution to the RDE as being located at the nodes of a
(family) tree (for a Galton-Watson branching process). Actually,
for some purposes we shall find it more convenient to embed this family
tree into $\tT$, the deterministic tree with infinite branching factor of size
$\aleph_0$.

The generic setup in such circumstances is to find distributional fixed
points of the recursion:
\begin{equation}
\label{Gen}
X_u = f(\xi_u;X_{ui}, i\geq 1),
\end{equation}
where $X_u$ and $\xi_u$ are respectively, the value and the noise associated with node 
$u$ and $ui$ is the address of the $i$th daughter of node $u$.

With this model, it is of some interest not only to find solutions to
the RDE (\ref{Gen}) but also to  answer the question of endogeny:
$$
\hbox{\lq is $(X_u; u\in
\tT)$ measurable with respect to $(\xi_u;u\in \tT)$?'}
$$
If this measurability condition holds, then $X_{\cdot}$
is said to be endogenous.

In the context of the error-reporting model, endogeny represents the worst possible
situation---the top-level error report is based entirely on the noise
and is uninfluenced by the error state of low-level sub-systems.
Similarly, in the veto-voter paradigm, endogeny represents the situatrion where the voice of
the \lq little man' is completely swamped by reversals by officials.

In this paper we will first show how to transform (\ref{DE1}) into the new
RDE:
\begin{equation}
\label{DE2} X = 1-\prod_{i=1}^N X_i,
\end{equation}
for a suitable, random variable $N$, independent of the $X_i$,
Then we'll not only find all the solutions to this RDE on $[0,1]$, their
basins of attractions and the limit cycles of the corresponding map on
the space of distributions on $[0,1]$, but also
give necessary and sufficient
conditions for the corresponding solutions on $\tT$ to be endogenous.

The fundamental technique we use, which we believe is entirely novel, is
to consider the distribution of a solution conditional  upon the noise and to identify
endogeny by showing that this conditional distribution is concentrated
on $\{0,1\}$.

\section{Notation and a transformation of the RDE}\label{tnote}
\subsection{Tree-indexed solutions}
We seek distributions $\nu$ on $[0,1]$ such that if ($Y_i; 1\leq i)$ are
independent with distribution $\nu$, then the random variable $Y$
satisfying (\ref{DE1}) also has distribution $\nu$. More precisely,
writing $\P$ for the set of probability measures on
$[0,1]$, suppose that $M$ has distribution $d$ on
$\mathbb{\overline{Z}}_+$ and define the map
$$\T \equiv \T_d : \P \rightarrow
\P$$
Then we set $\T(\nu)$ to be the law of the
random variable $Y$ given by (\ref{DE1}), when the $Y_i$ are
independent and identically distributed with distribution $\nu$ and are
independent of $N$,
and seek fixed points of the map $\T$. The existence and
uniqueness of fixed points of this type of map, together with
properties of the solutions, are addressed by Aldous and
Bandhapadhyay in \cite{Aldous} (the reader is also referred to 
\cite{Band} and \cite{Rusc} and the references therein).
The linear and min cases are particularly
well-surveyed, though we are dealing with a non-linear case to
which the main results do not apply.

A convenient
generalisation of the problem is the so-called {\em tree-indexed} problem, in which
we think of the $Y_i$ as being marks associated with the daughter nodes of the root of $T$, a family tree of a
Galton-Watson branching process. We start at some level $m$ of the
random tree. Each vertex $v$ in
level $m-1$ of the tree has $M_v$ daughter vertices, where the $M_v$
are i.i.d. with common distribution $d$ and 
has associated with it noise $\xi_v$, where the $(\xi_u;u\in T)$
are iid and are independent of the $(M_u;u\in T)$.

By associating with daughter
vertices independent random variables $Y_{vi}$ having
distribution $\nu$, we see that $Y_v$ and $Y_{vi};1\leq i\leq M_v$ satisfy
equation (\ref{DE1}).

In this setting the notion of endogeny was introduced in
\cite{Aldous}. Loosely speaking, a solution to the tree-indexed
problem (which we will define precisely in the next section) is
said to be endogenous if it is a function of the initial data or noise
alone so that no additional randomness is present.

It is convenient to work on a tree with infinite
branching factor and then think of the random tree of the previous
paragraph as being embedded within it. An initial ancestor (in
level zero), which we denote $\emptyset$, gives rise to a
countably infinite number of daughter vertices (which form the
members of the first generation), each of which gives rise to an
infinite number of daughters (which form the members of the second
generation), and so on. We assign each vertex an address according
to its position in the tree: the members of the first generation
are denoted $1, 2, ...$, the second $11,12,\ldots 21, 22, \ldots ,
31, 32, \ldots$ etc, so that vertices in level $n$ of the tree
correspond to sequences of positive integers of length $n$.
We also write $uj, j = 1, 2, ...$ for the daughters of a
vertex $u$. We write $\tT$ for the collection of all
vertices or nodes (i.e. $\tT = \bigcup_{n=0}^\infty
\mathbb{N}^n$) and think of it as being partitioned by depth,
that is, as being composed of levels or generations, in the way
described and define the depth function $|\cdot|$ by $|u| = n$ if vertex $u$ is in
level $n$ of the tree. Associated to each of the vertices $u \in
\tT$ are iid random variables $M_u$ with distribution $\d$, telling us the
(random) number of offspring produced by $u$. The vertices $u1,
u2, ..., uM_u$ are thought of as being alive (relative to
$\emptyset$) and the $\{uj: j> M_u\}$  as dead. We can now write our original equation as a
recursion on the vertices of $\tT$:
\begin{equation}
\label{RDE1} Y_u = \xi_u\prod_{i=1}^{M_u} Y_{ui}+(1-\xi_u)(1-\prod_{i=1}^{M_u} Y_{ui}), \; u
\in \tT.
\end{equation}
The advantage of the embedding now becomes clear: we can talk
about the RDE at any vertex in the infinite tree and yet, because the product only runs
over the live daughters relative to $u$, the random Galton-Watson family
tree is encoded into the RDE as noise.

\subsection{The transformed problem}
It is a relatively simple maer to transform the RDE (\ref{RDE1}) into
the following, simpler, RDE:
\begin{equation}\label{RDE}
X_u = 1-\prod_{i=1}^{N_u} X_{ui}, \; u
\in \tT.
\end{equation}
To do so, first note that if we colour red all the nodes, $v$, in the tree
$\tT$ for which $\xi_v=0$ then it is clear that we may proceed down each
line of descent from a node $u$ until we hit a red node. In this way, we
either "cut" the tree at a collection of nodes which we shall view as
the revised family of $u$, or not, in which case $u$ has an infinite
family. Denote this new random family size by $N_u$ then
$$Y_u= 1-\prod_{i=1}^{N_u} Y_{\hat{ui}},$$
if $u$ is red, where $\hat{ui}$ denotes the $i$th red node in the
revised family of $u$.
Now condition on node $u$ being red, then with this revised tree we
obtain the RDE (\ref{RDE}).
It is easy to see that if the original tree has family size PGF $G$,
then the family size in the new tree corresponds to the total number of
deaths in the original tree when it is independently thinned, with the descendants of each
node being pruned with probability $q$. It is easy to obtain the
equation for the PGF, $H$, of the family size $N_u$ on the
new tree:
\begin{equation}\label{trans}
H(z)=G(pH(z)+qz).
\end{equation}

\section{The discrete and conditional probability solutions}
We begin with some notation and terminology. We say that the
random variables in (\ref{RDE1}) are weakly stationary if $X_u$ has
the same distribution for every $u \in \tT$. The
stationarity of the $X_u$ corresponds to $X_u$ having as
distribution an invariant measure for the distributional equation
(\ref{RDE}). 
\begin{defn}
We say that the process (or collection of random
variables) $\X = (X_u; u \in \tT)$ is a {\em tree-indexed}
solution to the RDE (\ref{RDE}) if
\begin{enumerate}
\item for every $n$, the random variables $(X_u; |u| = n)$ are
mutually independent and independent of $(N_v; |v| \leq n-1)$;
\item for every $u \in \tT$, $X_u$ satisfies
\[X_u = 1 - \prod_{i=1}^{N_u} X_{ui},\] and the $(X_u; u \in
\tT)$ are weakly stationary. \end{enumerate}
\end{defn}

Notice that these conditions determine the law of $\X$. This
means that a tree-indexed solution is also stationary in the
strong sense, that is, a tree-indexed solution is ``translation
invariant" with respect to the root (if we consider the collection
$\X^v = (X_u; u \in \tT_v)$, where $\tT_v$ is the
sub-tree rooted at $v$, then $\X^v$ has the same distribution as
$\X$ for any $v \in \tT$). Furthermore, we say that such a
solution is {\em endogenous} if it is measurable with respect to the
random tree (i.e. the collection of family sizes) $(N_u; u \in
\tT)$. As we remarked in the introduction, in informal
terms this means that the solution depends only on the noise with no additional randomness coming from the boundary of
the tree. See \cite{Aldous} for a thorough discussion of endogeny
together with examples.
\newline
\newline
The following is easy to prove.
\begin{lem}
Let $(X_u; u \in \tT)$ be a tree-indexed solution to the
RDE (\ref{RDE}). Then the following are equivalent:
\begin{enumerate} 
\item $X$ is endogenous; 
\item $X_\emptyset$ is
measurable with respect to $\sigma(N_u; u \in \tT)$; 
\item $X_u$ is measurable with respect to $\sigma(N_v; v \in
\tT)$ for each $u \in \tT$; 
\item $X_u$ is measurable
with respect to $\sigma(N_v; v \in \tT_u)$ for each $u \in
\tT$.
\end{enumerate}
\end{lem}

\begin{remark} Notice that if a tree-indexed solution to (\ref{RDE}) is endogenous
then property (1) of a tree-indexed solution is
automatic: for every $u \in \tT$, $X_u$ is measurable with
respect to $\sigma(N_v; v \in \tT_u)$ and hence is
independent of $(N_v; |v|\leq n-1)$.
\end{remark}
\begin{lem} \label{invmsr}
There exists a unique probability measure on $\{0,1\}$ which is
invariant under (\ref{DE2}).
\end{lem}
\begin{proof} Let $X$ be a random variable whose distribution is
concentrated on $\{0,1\}$ and which is
invariant under (\ref{DE2}). Let ${\mu^1}= \mathbb{P}(X = 1)$. We
have then $\mathbb{P}(X = 0) = 1- {\mu^1}$ and 
\[
\mathbb{P}(X_i =
1;\hbox{ for } i = 1, ..., N) = \sum_n \mathbb{P}(X_i = 1;\hbox{ for }  i = 1, ...,
n|N=n)\mathbb{P}(N=n)
=H(\mu^1).
\] 
Now,
$X = 0$ if and only if $X_i = 1$ for $i = 1, ..., N$. Hence a
necessary and sufficient condition for invariance is
\begin{equation}\label{mu}
1 - {\mu^1}=H({\mu^1}).
\end{equation}
Now let
\[
K(x) \defto H(x) + x - 1.
\] 
Since $H$ is a generating function and
$H(0) = 0$, we have $K(0) = -1 < 0$ and $K(1)  > 0$ so that $K$
is guaranteed to have a zero in $(0,1)$, and it is
unique since the mapping $x \mapsto H(x) + x$ is strictly
increasing.
\end{proof}
We can now deduce that there exists a tree-indexed solution on
$\{0,1\}^{\tT}$ to the RDE (\ref{RDE}) by virtue of Lemma 6
of \cite{Aldous}.

\begin{thm}
Let $\tS = (S_u; u \in \tT)$ be a tree-indexed solution on
$\{0,1\}^{\tT}$ to the RDE (\ref{RDE}) (i.e. the $S_u$ have the invariant
distribution on the two point set $\{0,1\}$), which we will henceforth refer
to as the {\em discrete solution}. Let $C_u =
\mathbb{P}(S_u = 1|N_v; v \in \tT)$. Then $\C = (C_u; u \in
\tT)$ is the unique {\em endogenous} tree-indexed solution to the RDE.
\end{thm}
\begin{proof}
To verify the relationship between the random variables, we have,
writing $\bN= (N_u; u \in \tT)$
and 
$\bN_u=(N_v; v\in\tT_u)$,

\[
C_u = {\mathbb P}(S_u=1|\bN )=\mathbb{E}[1_{(S_u = 1)}|\bN ] =
\mathbb{E}[S_u|\bN ]
\]
\[=\mathbb{E}[1 - \prod_{i=1}^{N_u} S_{ui}|\bN]\]
\[=1 - \mathbb{E}[\prod_{i=1}^{N_u} S_{ui}|\bN]\]
\[= 1 - \prod_{i=1}^{N_u} \mathbb{E}[S_{ui}|\bN]\]
\[= 1 - \prod_{i=1}^{N_u} C_{ui},\] since the $S_{ui}$ are independent and $ \bN$ is strongly
stationary.
To verify stationarity, let 
\[
C_u^n = \mathbb{P}(S_u = 1|N_v; |v|\leq n).
\] 
Then the sequence $(C_u^n)_{n \geq 1}$ is a uniformly
bounded martingale and so converges almost surely and in $L^2$ to a  limit
which must in fact be $C_u$. Now, we can write $C_u^n$ as
\begin{eqnarray}\label{***}
C_u^n &=& 1 - \prod_{i_1=1}^{N_u} C_{ui_1}^n\\
&=& 1 - \prod_{i_1=1}^{N_u}
\left(1 -\prod_{i_2=1}^{N_{u{i_1}}}\bigl(...(1-\prod_{i_{n-1-|u|}=1}^{N_{ui_1i_2\ldots i_{n-2-|u|}}}
(1-(\mu^1)^{N_{ui_1i_2\ldots i_{n-1-|u|}}}))...\bigr)\right)\nonumber\\
& \rightarrow& C_u \hbox{ a.s.}.\nonumber\\
\nonumber
\end{eqnarray}
This corresponds to starting the distributional recursion at level $n$ 
of the tree with unit masses at $\mu^1$. Now, $(C_u^n; u \in \tT)$ is
stationary since each $C_u^n$ is the same function of ${\bN}_u$,
which are themselves stationary. Since $C_u$ is the (almost sure) limit
of a sequence of stationary random variables, it follows that ${\tC}=
(C_u; u \in \tT)$ is stationary. Notice that the
conditional probability solution, ${\tC}$, is automatically endogenous
since $C_u$ is $\sigma(N_v; v \in \tT_u$)-measurable for
every $u \in \tT$ and hence $(C_u; |u| = n)$ is independent
of $(N_u; |u| \leq n-1)$. The independence of
the collection $(C_u; |u| = n)$ follows from the fact that the
$((S_u,{\bN}_u); |u| = n)$ are independent  .

Finally, notice that if $(L_u; u\in\tT)$ solve the RDE (\ref{RDE}) and
are integrable then $m\defto{\mathbb E} L_u$ must satisfy (\ref{mu}) and
hence must equal $\mu^1$. It now follows, that
 $L_u^n\defto {\mathbb E}[L_u|N_v; |v| \leq n]=C_u^n$, since at depth $n$, $L_u^n=\mu^1$ so that
$L_u^n$ also satisfies equation (\ref{***}) and hence must equal $C_u^n$. Now
$L_u^n\rightarrow L_u$ a.s. and so, if $L$ is endogenous then it must equal
$C$. This establishes that $C$ is the unique endogenous solution.
\end{proof}

\begin{remark}
Notice that if ${\textbf S}$ is endogenous then ${\tC} = {\textbf S}$ almost surely so
that if ${\textbf S}$ and ${\tC}$ do not coincide then ${\textbf S}$ cannot be endogenous.
\end{remark}

\section{The moment equation and uniqueness of solutions}\label{momsec} Many of the results proved in this
paper rely heavily on the analysis of equation (\ref{moment}) below.

\begin{thm}
Any invariant distribution for the RDE (\ref{RDE}) must have
moments $(m_n)_{n \geq 0}$ satisfying the equation
\begin{equation} \label{moment} H(m_n) - (-1)^nm_n = \sum_{k=0}^{n-1}
(^n_k)(-1)^km_k,\end{equation} where $m_n^{1+1/n} \leq m_{n+1}
\leq m_n$ and $m_0 = 1$.
\end{thm}
\begin{proof}
Let $X$ be a random variable whose distribution is invariant for
the RDE and write $m_k = \mathbb{E}[X^k]$. Applying the RDE
(\ref{RDE}) to $(1-X)^n$ we have
\[\mathbb{E}[(1-X)^n] = \mathbb{E}[\prod_{i=1}^N X_i^n] = H(m_n).\]
On the other hand, by expanding $(1-X)^n$ we obtain
\[\mathbb{E}[(1-X)^n] = \mathbb{E}[\sum_{k=0}^n (^n_k)(-1)^k
X_i^k]\]
\[= \sum_{k=0}^n (^n_k)(-1)^k m_k,\] so that \[H(m_n) =
\sum_{k=0}^n (^n_k) (-1)^k m_k.\] The condition $m_{n+1} \leq m_n$
follows from the fact that the distribution is on $[0,1]$. The
other condition follows from the monotonicity of $L^p$ norms.
\end{proof}
As an example, if the random variable $N$ has generating function
$H(x) = x^2$ (i.e. $N \equiv 2$), the moment equation tells us
that
\[m_1^2 + m_1 - 1 = 0\] so that $m_1 = (\sqrt{5} - 1)/2$. For
$m_2$ we have \[m_2^2 - m_2 - (2 - \sqrt{5}) = 0\] so that $m_2 =
m_1$ or $m_1^2$ and so on. In fact the two possible moment
sequences turn out to be $m_0 = 1, m_n = (\sqrt{5}-1)/2$ for $n
\geq 1$ or $m_0 = 1, m_1 = (\sqrt{5} - 1)/2, m_n = m_1^n$ for $n
\geq 2$.
\newline
We suppose from now on that $H(0)=0$ and $H$ is strictly convex (so that
${\mathbb P}(2\leq N <\infty)>0$).
\newline We now state the main result of the paper.

\begin{thm} \label{main}
Let $S = (S_u; u \in \tT)$ and $C = (C_u; u \in
\tT)$ be, respectively, the discrete solution and
corresponding conditional probability solution to the RDE
(\ref{RDE}). Let $\mu^1 = \mathbb{E}[S_u]$. Then
\begin{enumerate}
\item S is endogenous if and only if $H'(\mu^1) \leq 1$; \item $C$
is the unique endogenous solution; \item the only invariant
distributions for the RDE (\ref{RDE}) are those of $S_\emptyset$
and $C_\emptyset$.
\end{enumerate}
\end{thm}
The proof of the theorem relies on several lemmas. For (1) we
extend a result of Warren \cite{Warren} by first truncating $N$
and then take limits.

First however, we give some consequences of the moment equation
(\ref{moment}):
\begin{lem}\label{momlem}
There are at most two moment sequences satisfying (\ref{moment}).
Moreover, the first moment $m^1$ is unique and equal to $\mu^1$,  $1>m^1>\frac{1}{2}$ and in the case
that $H'(m^1)\leq 1$ there is only one moment sequence satisfying
(\ref{moment}).

\end{lem}
\begin{proof}
Uniqueness of $\mu^1$ (the root of $f(m^1)=1$, where $f:t\mapsto H(t)+t$)
has already been shown in Lemma \ref{invmsr}. 
Now set
$$g(x)=H(x)-x,
$$
then $g$ is strictly convex on [0,1] with $g(0)=0$ and $g(1-)=H(1-)-1\leq 0$. Thus there are
at most two solutions of $g(x)=1-2m^1$. Since $m^1$ itself is a
solution,
it follows that $1-2m^1\leq 0$
and there is at most one other solution. There is another solution with
$m^2< m^1$ if and only if $m^1$ is greater than $\mu^*$, the argmin of $g$, and
this is clearly true if and only if $g'(m^1)>0\Leftrightarrow H'(m^1)>1$.

Suppose that this last inequality holds, so that there is a solution, $m^2$, of
$g(x)=1-2m^1$ with $m^2<\mu^*<m^1$. There is at most one solution of
$$
f(x)=1-3m^1+3m^2,
$$
and if it exists take this as $m^3$.
Similarly,
there is at most one solution of
$g(x)=1-4m^1+6m^2-4m^3$
to the left of $\mu^*$ and this is the only possibility for $m^4$.
Iterating the argument, we obtain at most one strictly decreasing sequence
$m^1,\ldots$.

\end{proof}

\subsection{The case of a bounded branching factor}
Recall that the random family size $N$ may take the value $\infty$.
\begin{lem} \label{properties} 
Define $N^n = min(n,N)$ and denote its generating function by $H_n$.
Then $N^n$ is bounded and 
\begin{enumerate} 
\item $H_n(s) \geq H(s)$ for all $s \in [0,1]$;
\item $H_n \rightarrow H$ uniformly on compact subsets of $[0,1)$; 
\item $H_n'\rightarrow H'$ uniformly on compact subsets of $[0,1)$.
\end{enumerate}
\end{lem}
We leave the proof to the reader.

The following lemma will be used in the proof of Theorem
\ref{conv}.
\begin{lem}\label{prop2}
Let $C_u^{(n)} = \mathbb{P}(S_u = 1|N^n_u; u \in \tT)$ denote
the conditional probability solution for the RDE (\ref{RDE}) with $N$ replaced by $N^n$. Let $\mu_n^k =
\mathbb{E}[(C_u^{(n)})^k]$ denote the corresponding $k$th moment and
let $\mu^k = \mathbb{E}[(C_u)^k]$. Let $\mu^*_n$ denote the argmin of
$g_n(x)\defto H_n(x)-x$ and let $\mu^2_{n,m}$ denote that root
of the equation,
\begin{equation} \label{modified} 
g_n(x) = 1 - \mu^1_m - \mu^1_n, 
\end{equation} 
which lies to the left of $\mu^*_n$ (i.e. the lesser of the two possible
roots). Then $\mu_n^k \rightarrow \mu^k$ for $k = 1,2$ and
$\mu^2_{n,m} \rightarrow \mu^2$ as $\min(n,m)\rightarrow \infty$.
\end{lem}
\begin{proof}
For the case $k = 1$, consider the graphs of the functions $H_n(x)
+ x$ and $H(x) + x$. We have $H_n(x) \geq H(x)$ for all $x \geq 0$
and for all $n \geq 1$ so that $\mu_n^1$ is bounded above by
$\mu^1$ for every $n$, since $\mu^1_n$ and $\mu^1$ are respectively the
roots of 
$$
H_n(x)+x=1\hbox{ and }H(x)+x=1.
$$
Furthermore, since $H_n$ decreases to $H$
pointwise on $[0,1)$, it follows that the $\mu_n^1$ are
increasing. The $\mu_n^1$ must therefore have a limit, which we
will denote $\widehat{\mu}$. 

It follows from Lemma \ref{properties} that, since $\mu^1<1$,
$H_n(\mu_n^1) \rightarrow H(\widehat{\mu})$. Hence \[1 =
H_n(\mu_n^1) + \mu_n^1 \rightarrow H(\widehat{\mu}) +
\widehat{\mu},\] so that $\widehat{\mu}$ is a root of $H(x) + x =
1$. It follows, by uniqueness, that $\widehat{\mu} = \mu^1$. 
\newline
\newline
For the case $k = 2$
we consider the graphs of $g_n(x)$ and $g(x)$. 
We first show that $\mu_n^2 \rightarrow \mu^2$ and then that
$\mu_{n,m}^2 \rightarrow \mu^2$ as $\min(n,m) \rightarrow \infty$.
\newline
\newline To show that $\mu_n^2 \rightarrow \mu^2$ we argue that
$\mu^2$ is the only limit point of the sequence $(\mu_n^2)_{n \geq
1}$. Notice that, since $\mu_n^1 \rightarrow \mu^1$ and $\mu_n^2$
satisfies
\[H_n(\mu_n^2) - \mu_n^2 = 1 - 2\mu_n^1,\] the only possible limit
points of the sequence $(\mu_n^2)_{n \geq 1}$ are $\mu^1$ and
$\mu^2$. Now, either $\mu^1\leq \mu^*$, in which case $\mu^1=\mu^2$ or,
$\mu^2 \leq \mu^* < \mu^1<1$. In the latter case, it is easy to show
that $\mu^*_n\rightarrow \mu^*$ (by uniform continuity of $g_n'$) and
so, since $\mu^1_n\rightarrow \mu^1$ it follows that 
$$
\mu^1_n>\mu^*_n,
$$
for sufficiently large $n$, and hence 
$$\mu^2_n\leq \mu^*_n,
$$
for sufficiently large $n$. In either case,
the only possible limit point is $\mu^2$; since the
$\mu_n^2$ are bounded they must, therefore, converge to $\mu^2$.
\newline \newline 
We conclude the proof by showing that $\mu^2$ is
the only limit point of the sequence $(\mu^2_{n,m})$. Since
$\mu_m^1, \mu_n^1 \rightarrow \mu^1$ as $\min(n,m) \rightarrow
\infty$ and $\mu_{n,m}^2$ satisfies (\ref{modified}), the only
possible limit points of the sequence $(\mu_{n,m}^2)_{m,n \geq 1}$
are $\mu^1$ and $\mu^2$. 

Once more, consider the two cases:
$$
\mu^1\leq\mu^*\hbox{ and }\mu^1>\mu^*.
$$
In the first case, $\mu^1_n=\mu^2_n$, for sufficiently large $n$, so
that $\mu^2$ is the only limit point;
in the second case
$$
\mu^1=\liminf_n \mu^1_n> \mu^*=\limsup_n\mu^*_n,
$$
and since $\mu^2_{n,m}\leq \mu^*_n$, $\mu^1$ cannot be a limit point.
Thus, in either case, $\mu^2$ is the unique limit point and hence is the
limit.

\end{proof}
\begin{remark}Notice that the method of the proof can be extended to prove
that $\mu_n^k \rightarrow \mu^k$ for any $k$. 
\end{remark}

\begin{thm}\label{conv}
$C_u^{(n)}$ converges to $C_u$ in $L^2$.
\end{thm} 
\begin{proof}
Let $n \geq m$. Define $E_{m,n} = \mathbb{E}[(C_u^{(m)} - C_u^{(n)})^2]$.
Expanding this, we obtain
\[E_{m,n} = \mu_m^2 + \mu_n^2 - 2r_{m,n},\] where $r_{m,n} = \mathbb{E}[C_u^{(m)}C_u^{(n)}]$. On the other hand,
by applying the RDE (\ref{RDE}) once, we obtain 
\[E_{m,n} =
\mathbb{E}[(\prod_{i=1}^{N_u^n} C_{ui}^{(n)} -
\prod_{i=1}^{N_u^m}C_{ui}^{(m)})^2]\] 
\[= H_m(\mu_m^2) + H_n(\mu_n^2)
- 2\mathbb{E}[\prod_{i=1}^{N_u^m} C_{ui}^{(m)}\prod_{i=1}^{N_u^n}
C_{ui}^{(n)}].\] 
We can bound $E_{m,n}$ above and
below as follows: since each $C^k_{ui}$ is in [0,1] omitting terms from
the product above increases it, while adding terms decreases it. Thus,
since $n \geq m$, $N^n_u\geq N^m_u$, and so
replacing $N^n_u$ by $N^m_u$ in the product above increases it while
replacing $N^m_u$ by $N^n_u$ decreases it. Thus we get:
\[
H_m(\mu_m^2) + H_n(\mu_n^2) - 2H_m(r_{m,n}) \leq E_{m,n} \leq
H_m(\mu_m^2) + H_n(\mu_n^2) - 2H_n(r_{m,n}).
\] 
Using the upper
bound we have 
\[
2H_n(r_{m,n}) \leq H_m(\mu^2_m) + H_n(\mu^2_n) -
E_{m,n} = H_m(\mu^2_m) + H_n(\mu^2_n) - \mu^2_m - \mu^2_n +
2r_{m,n}.
\] 
The moment equation (\ref{moment}) tells us that
$H_m(\mu^2_m) - \mu^2_m = 1 - 2\mu^1_m$ and that $H_n(\mu^2_n) -
\mu^2_n = 1 - 2\mu^1_n$. 
Hence 
\[2H_n(r_{m,n}) \leq 1 - 2\mu^1_m +
\mu^2_m + 1 - 2\mu^1_n + \mu^2_n - \mu^2_m - \mu^2_n + 2r_{m,n},
\]
so that, on simplifying, 
\[H_n(r_{m,n}) - r_{m,n} \leq 1 - \mu^1_m
- \mu^1_n.\] 
Recall that the equation $H_n(x) - x = 1 - \mu^1_m -
\mu^1_n$ has (at most) two roots, the lesser of which we denoted
$\mu^2_{m,n}$. Let $\mu^1_{m,n}$ be the other (larger) root (or 1, if the second root does not exist). 
Then, since
$H_n(x) - x$ is convex,
$\mu_{n,m}^2 \leq r_{m,n} \leq \mu_{n,m}^1$ for all $m,n$ and
hence $\liminf_{m \rightarrow \infty} r_{m,n} \geq \mu^2$ since
$\mu_{n,m}^2 \rightarrow \mu^2$ by Lemma \ref{prop2}.
\newline \newline 
On the other hand, Holder's inequality tells us
that $r_{m,n} \leq \sqrt{\mu_m^2\mu_n^2}$ and so it follows that
$\limsup_{m \rightarrow \infty} r_{m,n} \leq \mu^2$ since
$\mu_m^2, \mu_n^2 \rightarrow \mu^2$ by Lemma \ref{prop2}. Hence
$r_{m,n} \rightarrow \mu^2$ as $n \rightarrow \infty$ and
\[E_{m,n} \rightarrow \lim_{m,n \rightarrow \infty} \mu_m^2 +
\mu_n^2 - 2r_{m,n} = \mu^2 + \mu^2 - 2\mu^2 = 0,\] showing that
$(C_u^{(n)})$ is Cauchy in $L^2$. It now follows, by the completeness
of $L^2$, that $C_u^{(n)}$ converges. Since $C_u^{(n)}$ is
$\sigma(N)$-measurable, the limit $L_u$ of the $C_u^{(n)}$ must also be
$\sigma(N)$-measurable for each $u$ and the collection $(L_{ui})_{i\geq 1}$ must be independent and 
identically distributed on [0,1], with common mean $\mu^1<1$. Moreover,
by strong stationarity of the $C^{(n)}$s, the  $L_u$s are strongly
stationary.

To verify that $L_{\emptyset}$ is the conditional
probability solution, notice that
\[1_{E_n} C_\emptyset^{(n)} = (1 - \prod_{i=1}^{N_\emptyset^n} C_i^{(n)}) 1_{E_n}\]
\[= (1 - \prod_{i=1}^{N_\emptyset} C_i^{(n)}) 1_{E_n},\] 
where
$E_n = \{N_\emptyset \leq n\}$. As $n \rightarrow \infty$, $E_n\uparrow E\defto (N<\infty)$; 
furthermore, since the $C_i^{(n)}$
converge in $L^2$, they do so in probability. We may assume without loss of generality, therefore,
that $C^{(n)}_i$ converges almost surely for each $i$ so that, in
the limit,
\begin{equation}\label{endo2}
1_E L_\emptyset = \lim 1_{E_n}C^{(n)}_\emptyset =\lim 1_{E_n}
(1 - \prod_{i=1}^{N_\emptyset} C^{(n)}_i)=1_E(1-\prod_{i=1}^{N_\emptyset}L_i) \hbox{ a.s.}
\end{equation}
It is easy to show that 
$$\prod_{i=1}^\infty L_i=0\hbox{ a.s.}
$$
while
$$1_{E^c}C^{(n)}_\emptyset=1_{E^c}(1-\prod_{i=1}^{(n)}C^{(n)}_i) \rightarrow
1_{E^c}\hbox{ a.s.},
$$
so that 
\begin{equation}\label{endo3}
1_{E^c}L_\emptyset=\lim 1_{E^c}C^{(n)}_\emptyset=1_{E^c}.
\end{equation}
Thus 
adding equations (\ref{endo2}) and (\ref{endo3}) we see that
$$
L_\emptyset=(1-\prod_{i=1}^{N_\emptyset}L_i),$$
and so
$L$ is an endogenous solution to the RDE. It follows from uniqueness that $L$
must be the conditional probability solution $C$.
\end{proof}
\subsection{Proof of Theorem \ref{main}}We are now nearly in a position to finish proving Theorem \ref{main}.
To recap, we have shown in Lemma \ref{momlem} that there are at most two
distributions which solve the RDE (\ref{DE2}), corresponding to the `moment
sequences' $\mu^1,\mu^1,\ldots$ and $\mu^1,\mu^2,\ldots$. The first of
these is the moment sequence corresponding to the distribution on
$\{0,1\}$ with mass $\mu^1$ at 1. The second may or may not be a true
moment sequence and is equal to the first if and only if $H'(\mu^1)\leq
1$. Moreover, there is only one endogenous solution, and this
corresponds to the conditional probability solution $C$, thus if we can show
that $C$ is not discrete (i.e. is not equal to $S$) whenever
$H'(\mu^1)> 1$ then we will have proved the result.

We need to recall
some theory from \cite{Warren}. Consider the recursion \[\xi_u =
\phi(\xi_{u0}, \xi_{u1},..., \xi_{u(d-1)}, \epsilon_u), \ \ \ \ u
\in \Gamma_d,\] where the $\xi_u$ take values in a finite space
${\mathcal S}$, the ``noise" terms $\epsilon_u$ take values in a space $E$,
$\Gamma_d$ is the deterministic $d$-ary tree and $\phi$ is
symmetric in its first $d-1$ arguments. We suppose that the
$\epsilon_u$ are independent with common law $\nu$ and that there
exists a measure $\pi$ which is invariant for the above recursion
(i.e. $\pi$ is a solution of the associated RDE). Let $u_0 =
\emptyset, u_1, u_2, ...$ be an infinite sequence of vertices
starting at the root, with $u_{n+1}$ being a daughter of $u_n$
for every $n$. For $n \leq 0$, define $\xi_n = \xi_{u_{-n}}$. Then,
under the invariant measure $\pi$, the law of the sequence $(\xi_n;
n \leq 0)$, which, by the symmetry of $\phi$ does not depend on
the choice of sequence of vertices chosen, is that of a stationary
Markov chain. Let $P^2$ be the transition matrix of a Markov chain
on ${\mathcal S}^2$, given by
\[P^2((x_1,x_1'), A \times A') = \int_{\mathcal S} \int_E
1_(\phi(x_1,x_2, ..., x_d, z) \in A, \phi(x_1', x_2, ...,
x_d, z) \in A') d\nu(z)d\pi(x_2)...d\pi(x_d).\] Let $P^-$ be the
restriction of $P^2$ to non-diagonal terms and $\rho$ the Perron-
Frobenius
eigenvalue of the matrix corresponding to $P^-$. 
\newline \newline
The following theorem gives a necessary and sufficient condition
for endogeny of the tree-indexed solution corresponding to $\mu$. 
This is a small generalisation of Theorem 1 of \cite{Warren}
\begin{thm} 
\label{endthm1} 
The tree-indexed solution to the RDE associated
with
\[\xi_u = \phi(\xi_{u0}, \xi_{u1},..., \xi_{u(d-1)}, \epsilon_u),\]
corresponding to the invariant measure $\pi$, is endogenous if $d\rho
< 1$; it is {\em non-}endogenous if $d\rho
>1$. 
In the critical case $d\rho = 1$, let $\mathcal{H}_0$
be the collection of $L^2$ random variables measurable with respect to
$\xi_\emptyset$ and let $\mathcal{K}$ denote the $L^2$ random
variables measurable with respect to $(\epsilon_u; u \in
\Gamma_d)$. Then endogeny holds in this case provided $P^-$ is
irreducible and $\mathcal{H}_0 \cap \mathcal{K}^\perp = \{0\}$.
See \cite{Warren} for full details.
\end{thm}

\begin{thm} \label{endthm2}
Consider the RDE \begin{equation} \label{truncRDE} 
X_u = 1 -
\prod_{i=1}^{N^n_u} X_{ui}.
\end{equation} 
Then, by Lemma
\ref{invmsr}, there exists an invariant probability measure on
$\{0,1\}$ for (\ref{truncRDE}). Let $\mu_n^1$ denote the
probability of a 1 under this invariant measure. Then the
corresponding tree-indexed solution is endogenous if and only if
$H_n'(\mu_n^1) \leq 1$.
\end{thm}

\begin{proof}
Let $N^* = ess \sup N < \infty$ be a bound for $N$. We can then
think of the random tree with branching factor $N$ as being
embedded in a $N^*$-ary tree. Each vertex has $N^*$ daughter
vertices and the first $N$ of these are thought of as being
{\em alive} (the remaining being {\em dead}). In this context our RDE
reads 
\[X = 1 - \prod_{live \ u} X_u.\] 
We now compute the
transition probabilities from the previous theorem. Consider first
the transition from $(0,1)$ to $(1,0)$. The first coordinate
automatically maps to 1 and the second maps to 0 provided all of
the inputs not on the distinguished line of descent are equal to
1. The conditional probability of the vertex on the distinguished
line of descent being alive is $N/N^*$ since there are $N^*$
vertices, of which $N$ are alive. The probability of the remaining
$N-1$ vertices each taking value 1 is $(\mu_n^1)^{N-1}$ and so the
probability of a transition from $(0,1)$ to $(1,0)$, conditional on $N$,
is just
\[
1_{(N \geq 1)} \frac{(\mu_n^1)^{N-1}N}{d}.
\] 
Taking expectations, the required probability is 
\[
\mathbb{E}[1_{(N \geq 1)} \frac{(\mu_n^1)^{N-1}N}{N^*}] =
\frac{\mathbb{E}[1_{(N \geq 1)} N (\mu_n^1)^{N-1}]}{N^*} =
\frac{H_n'(\mu_n^1)}{N^*}.
\] 
The probability of a transition from
$(1,0)$ to $(0,1)$ is the same by symmetry. Hence $P^-$ is given by 
\[P^- = \left(%
\begin{array}{cc}
  0 & \frac{H_n'(\mu_n^1)}{N^*} \\
  \frac{H_n'(\mu_n^1)}{N^*} & 0 \\
\end{array}%
\right),\] and the Perron-Frobenius eigenvalue $\rho$ is
$\frac{H_n'(\mu_n^1)}{N^*}$. By Theorem \ref{endthm1}, the
criterion for endogeny is $N^*\rho \leq 1$, i.e. $H_n'(\mu_n^1)
\leq 1$, provided that, in the critical case $H_n'(\mu^1_n) = 1$,
we verify the stated non-degeneracy conditions. \newline
\newline It is easily seen that $P^-$ is irreducible. For the other
criterion, let $X \in \mathcal{H}_0 \cap \mathcal{K}^\perp$ so
that $X = f(X_\emptyset)$ for some $L^2$ function $f$ and
$\mathbb{E}[XY] = 0$ for all $Y \in \mathcal{K}$. Taking $Y = 1$,
we obtain $\mathbb{E}[X] = 0$. Writing $X$ as \[X = a
1_{(X_\emptyset = 1)} + b1_{(X_\emptyset = 0)},\]
where $a,b$ are constants, we obtain \[X = a
1_{(X_\emptyset = 1)} - \frac{a \mu_n^1}{1-\mu_n^1}
1_{{1}(X_\emptyset = 0)}.\] For convenience we will scale by
taking $a = 1$ (we assume that $X \neq 0$): \[X =
1_{(X_\emptyset = 1)} - \frac{ \mu_n^1}{1-\mu_n^1}
1_{(X_\emptyset = 0)}.
\] 
Now, for each $k$ take $Y_k =
1_{(N_\emptyset = k)} \in \mathcal{K}$. Then
\[
\mathbb{E}[XY_k] = \mathbb{E}[1_{(N_\emptyset =k)}
(1_{(X_\emptyset = 1)} - \frac{\mu_n^1}{1-\mu_n^1}
1_{(X_\emptyset = 0)})]
\] 
\[ 
= \mathbb{P}(N = k)[1 - (\mu_n^1)^k - \frac{(\mu_n^1)^{k+1}}{1 -\mu_n^1}]
\]
\[ 
=\mathbb{P}(N = k)(1 - \frac{(\mu_n^1)^k}{1 -\mu_n^1}) .
\] 
Now if we sum this expression over $k$ we get $1-\frac{H_n(\mu^1_n)}{1-\mu_n^1}=0$.
So either each term in the sum is zero or one or more are
not. But at least two of the probabilities are non-zero by assumption
(at least for sufficiently large $n$)
whilst the term $(1 - \frac{(\mu_n^1)^k}{1 -
\mu_n^1})$ can only disappear for at most one choice of $k$. Hence at least one of the terms is non-zero and
this contradicts the assumption that $X\in \mathcal{H}_0 \cap \mathcal{K}^\perp$.
\end{proof}

{\em Proof of the remainder of Theorem \ref{main}} We prove that $H'(\mu^1) > 1$ implies $\textbf S$ is not
endogenous so that $\tC$ cannot equal $\textbf S$. 
\newline \newline By Theorem \ref{endthm2} we know
that the RDE (\ref{truncRDE}) has two invariant distributions if
and only if $H_n'(\mu^1_n) > 1$. But we know that $C_u^{(n)}$
converges to $C_u$ in $L^2$ and hence $\mu_n^2 \rightarrow \mu^2
\neq \mu^1$ so that $S_u$ and $C_u$ have different second moments. It
now follows that $S_u$  does not have
the same distribution as $C_u$.  Since $[0,1]$ is bounded, this sequence
of moments determines a unique distribution which is therefore
that of $C_u$: see Theorem 1 of Chapter VII.3 of Feller \cite{Book}
\hfill$\square$

\section{Basins of attraction} Now we consider the {\em basin of
attraction} of the endogenous solution. That is, we ask for what
initial distributions does the corresponding solution at root,
$X_\emptyset$, converge (in law) to the endogenous solution.
\begin{defn}
Let $\varsigma$ be the law of the endogenous solution. Suppose
that we insert independent, identically distributed random
variables with law $\nu$ at level $n$ of the tree and apply the
RDE to obtain the corresponding solution $X_u^n(\nu)$ (with law
$\T^{n-|u|}(\nu)$) at vertex $u$. 

The basin of
attraction $B(\pi)$ of any solution is given by 
\[B(\pi) = \{\nu \in
\P : \T^n(\nu) \weak \pi\},\] which is,
of course, equivalent to the set of distributions $\nu$ for which
$X_u^n(\nu)$ converges in law to a solution $X$ of the RDE, with law $\pi$.
\end{defn}
\subsection{The unstable case: $H'(\mu^1) > 1$}
\begin{lem}
Suppose that $H'(\mu^1) > 1$.
Then
$X_u^n(\nu) \Lto C_u$, the
endogenous solution, for any $\nu$ with mean $\mu^1$ other than the
discrete measure on
 $\{0,1\}$.
\end{lem}
\begin{proof}
Let $E_k = \mathbb{E}[X_u^n(\nu)^2]$, where $k = n - |u|$, and let
$r_k = \mathbb{E}[C_u X_u^n(\nu)]$. Then
\[\mathbb{E}[(X_u^n(\nu) - C_u)^2] = E_k - 2r_k + \mu^2.\]
Now, \[E_k = \mathbb{E}[(1 - 2\prod_{i=1}^{N_u} X_{ui}^n(\nu) +
\prod_{i=1}^{N_u} X_{ui}^n(\nu)^2)]\] \[= 1 - 2H(\mu^1) +
H(E_{k-1}).\] This is a recursion for $E_k$ with at most two fixed
points (recall that the equation $H(x) - x = constant$ has at most
two roots). Recalling the moment equation (\ref{moment}), these
are easily seen to be $\mu^1$ and $\mu^2$, the first and second moments
of the endogenous solution. We have assumed that $\nu$ is not the
discrete distribution and so its second moment (i.e. $E_0$) must
be strictly less than $\mu^1$. Now, under the assumption that
$H'(\mu^1)>1$, $\mu^1$ and $\mu^2$ lie either side
of the minimum $\mu^*$ of $H(x) - x $ and
$H'(\mu^*) = 1$ so that $H'(\mu^2) < 1$. Hence $\mu^2$ is the stable
fixed point
and it now follows that $E_k$ converges to $\mu^2$. 
\newline
\newline The recursion for $r_k$ is essentially the same as that
for $E_k$: \[\mu^2 - r_k = H(\mu^2) - H(r_{k-1}).\] This has
$\mu^1$ and $\mu^2$ as  fixed points and, since \[r_0 = \mathbb{E}[C_u
X_u(\nu)] \leq \sqrt{\mathbb{E}[C_u^2] \mathbb{E}[X_u(\nu)^2]} <
\sqrt{\mu^1 \mu^1} = \mu^1,\] we are in the same situation as with
$E_k$. That is, we start to the left of $\mu^1$ and, because
$H'(\mu^1) > 1$, we conclude that $\mu^1$ is repulsive and
it follows that $r_k$ converges to $\mu^2$ under the
assumptions of the lemma. Hence
\[\mathbb{E}[(X_u^n(\nu) - C_u)^2] = E_k - 2r_k + \mu^2
\rightarrow 0.\]
\end{proof}
\begin{thm} Let $\delta$ denote the discrete distribution on
$\{0,1\}$ with mean $\mu^1$. Then
\[B(\varsigma) = \{\nu \in \P: \int x d\nu(x) = \mu^1 \ and \ \nu \neq
\delta\}.\] That is, $B(\varsigma)$ is precisely the set of
distributions on $[0,1]$ with the correct mean (except the
discrete distribution with mean $\mu^1$).
\end{thm}
\begin{proof}
We have already shown that \[\{\nu \in \P: \int x d\nu(x)
= \mu^1 \ and \ \nu \neq \delta)\} \subseteq B(\varsigma).\] 
Since the
identity is bounded on $[0,1]$, we conclude that
\[\mathbb{E}X_u^n(\nu) \rightarrow \mathbb{E}C_u,\hbox{ if }\nu\in
B(\varsigma),
\] so that $\nu \in B(\varsigma)$ only if the mean of $\T^n(\nu)$ converges to $\mu^1$.
From the moment equation (\ref{moment}), the mean of $X_u^n(\nu)$
is obtained by iterating the map $f$
$n$ times,
starting with the mean of $\nu$. This mapping has a unique fixed
point $\mu^1$ and, since $H'(\mu^1) > 1$, it is repulsive. It
follows that the only way we can have convergence in mean is if we
start with the correct mean, that is, if $\nu$ has mean $\mu^1$.
Hence
\[B(\varsigma) \subseteq \{\nu \in \P: \int x d\nu(x) = \mu^1\ \ and \ \nu \neq \delta\}.\]
\end{proof}
\subsection{The stable case: $H'(\mu^1) \leq 1$}
\begin{thm} \label{stable}
Let $b(\mu^1)$ be the basin of attraction of $\mu^1$ under the
iterative map for the first moment, $f:t \mapsto 1 - H(t)$. Then 
\[
B(\varsigma)
= \{\nu \in \P : \int xd\nu(x) \in b(\mu^1)\}.
\]
\end{thm}
Consider once again $\mathbb{E}[(X_u^n(\nu) - C_u)^2]$. Let
$m_k^\theta = \mathbb{E}X_u^n(\nu)^\theta$, where $k = n - |u|$.
Then
\[m^2_k = \mathbb{E}(1 - 2\prod_{i=1}^{N_u} X_{ui}^n(\nu) +
\prod_{i=1}^{N_u} X_{ui}^n(\nu)^2)\]
\[= 1 - 2H(m^1_{k-1}) + H(m^2_{k-1}).\] 
Recalling that $r_k = \mathbb{E}[C_u X_u^n(\nu)]$, we have
\[r_k = \mathbb{E}[(1 -
\prod_{i=1}^{N_u} C_{ui})(1 - \prod_{i=1}^{N_u} X_{ui}^n(\nu))]\]
\[= \mathbb{E}[(1 - \prod_{i=1}^{N_u} C_{ui} - \prod_{i=1}^{N_u} X_{ui}^n(\nu) +
\prod_{i=1}^{N_u} C_{ui}X_{ui}^n(\nu))]\] \[= 1 - H(\mu^1) - H(m^1_{k-1}) +
H(r_{k-1}).\] We now turn our attention to analysing the dynamics
of $m^2_k$ and $r_k$. We will concentrate on the equation for
$m^2_k$ as the equation for $r_k$ is essentially the same. By
assumption, $m^1_k$ converges to $\mu^1$ and so we may approximate
$m^1_k$, for $k \geq k_\epsilon$ (say), by $\mu^1 \pm \epsilon$,
for some small $\epsilon > 0$.
\begin{lem}
The trajectory $l_k$ of the dynamical system defined by the recursion 
\[
l_k = 1 -
2H(\mu^1 + \epsilon) + H(l_{k-1}), \;\;l_{k_\epsilon} =
m^2_{k_\epsilon},
\] 
is a lower bound for $m^2_k$ for all $k \geq
k_\epsilon$, where $k_\epsilon$ is a positive integer chosen so that
\[|m^1_k - \mu^1| < \epsilon, \hbox{ for } k \geq k_\epsilon.\]
\end{lem}
The proof is obvious.

\begin{lem}
Let \[f_\epsilon(x) = 1 - 2H(\mu^1 + \epsilon) + H(x), \ \ \ \ \ \
x \in [0,1].\] Then, for sufficiently small $\epsilon > 0$,
$f_\epsilon$ has a unique fixed  point $\mu^1(\epsilon)$ for which
$\mu^1(\epsilon)<\mu^*$. Moreover, as $\epsilon\rightarrow 0$,
$\mu^1(\epsilon)\rightarrow \mu^1$.
\end{lem}
\begin{proof}
This follows from uniform continuity, the fact that
$H(\mu^1+\epsilon)<H(\mu^1)$ and the the fact that $H'(\mu^1)\leq 1\Rightarrow \mu^1\leq\mu^*$.
\end{proof}

\begin{lem}
$l_k$ converges to $\mu^1(\epsilon)$.
\end{lem}
\begin{proof}
We have $l_k = f_\epsilon^{k-k_\epsilon}(l_{k_\epsilon})$ and so
we need only verify that $l_{k_\epsilon}$ is in the basin of
attraction of $\mu^1(\epsilon)$ and that $\mu^1(\epsilon)$ is
stable. We know that \[f_\epsilon(\mu^1 + \epsilon) < \mu^1 +
\epsilon\] since $1 - H(\mu^1 + \epsilon) < 1 - H(\mu^1) = \mu^1$
and so it must be the case that $\mu^1 + \epsilon \in
(\mu^1(\epsilon), p(\epsilon))$. It now follows that
$l_{k_\epsilon} < p(\epsilon)$ since $l_{k_\epsilon} \leq
m^1_{k_\epsilon} < \mu^1 + \epsilon$. In the strictly stable case
$H'(\mu^1) < 1$, the stability of $\mu^1(\epsilon)$ follows from
the fact that $\mu^1(\epsilon)$ converges to $\mu^1$ as $\epsilon$
tends to zero (by the previous lemma) and therefore
$\mu^1(\epsilon)$ can be made arbitrarily close to $\mu^1$ by
choosing $\epsilon$ to be sufficiently small. This means that for
sufficiently small $\epsilon$, $H'(\mu^1(\epsilon)) < 1$ by the
continuity of $H'$. In the critical case $H'(\mu^1) = 1$, we have
$\mu^1(\epsilon) < \mu^1 $, so that (by strict convexity)
$H'(\mu^1(\epsilon)) < 1$. In either case it now follows that
$f_\epsilon^{k-k_\epsilon}(l_{k_\epsilon})$ converges to
$\mu^1(\epsilon)$.
\end{proof}
{\em Proof of Theorem \ref{stable}}
 The preceding lemmas tell us that
\[\liminf_{k \rightarrow \infty} m^2_k \geq \lim_{k \rightarrow
\infty} l_k = \mu^1(\epsilon).\] Letting $\epsilon$ tend to zero,
we obtain \[\liminf_{k \rightarrow \infty} m^2_k \geq \mu^1.\] The
fact that $m^2_k \leq m^1_k$ for every $k$ gives us the
corresponding inequality for the lim sup:
\[\limsup_{k \rightarrow \infty} m^2_k \leq \lim_{k \rightarrow
\infty} m^1_k = \mu^1.\] We conclude that $m^2_k$ converges to
$\mu^1$. \newline \newline Now,
\[\mathbb{E}[(X_u^n(\nu) - C_u)]^2 = m^2_k - 2r_k + \mu^2,\] so
that $\mathbb{E}[(X_u^n(\nu) - C_u)^2] \rightarrow 0$, remembering
that in the stable case the discrete solution and endogenous
solution coincide (i.e. $\mu^1 = \mu^2$). We have now shown that
\[\{\nu \in \P : \int x d\nu(x) \in b(\mu^1)\} \subseteq
B(\varsigma),\] and the necessity for convergence in mean ensures that we
have the reverse inclusion. This completes the proof.
\hfill$\square$

\section{Outside the basin of attraction of the endogenous solution}
In this section we examine what happens if we iterate
distributions with mean outside the basin of attraction of the
endogenous solution.
\begin{defn} Recall that a map $f$ has an
$n$-cycle starting from $p$ if $f^n(p) = p$, where $f^n$ denotes the $n$-fold
composition of $f$ with itself.
\end{defn}
It is easily seen that the map for the first moment $f: t \mapsto 1 -
H(t)$ can have only one- and two-cycles. This is because the
iterated map $f^{(2)}:t \mapsto 1 - H(1-H(t))$ is increasing in $t$ and
hence can have only one-cycles. Notice also that the fixed
points (or one-cycles) of $\ft$ come in pairs: if $p$ is a
fixed point then so too is $1 - H(p)$.
\newline 
\newline
We consider the iterated RDE:
\begin{equation}\label{RDE3}
X=1-\prod_{i=1}^{N_\emptyset}(1-\prod_{j=1}^{N_i}X_{ij}).
\end{equation}
This corresponds to the iterated map on laws on [0,1], $\T^2$, where $\T$
is given at the beginning of section \ref{tnote}. We denote a generic
two-cycle of the map $\ft$ by the pair $(\mu^1_+,\mu^1_-)$.

\begin{thm}
Suppose that $(\mu^1_+,\mu^1_-)$ is a two-cycle of $\ft$. There are at
most two solutions of the RDE (\ref{RDE3}) with mean $\mp$. There is a
unique endogenous solution $C^+$, and a (possibly distinct) discrete
solution, $S^+$,
taking values in $\{0,1\}$. The endogenous solution $C^+$ is given by $P(S^+=1|\tT)$ 
(just as in the non-iterated case). The
solutions are distinct if and only if $H'(\mm)H'(\mp)>1$, i.e. if and
only if $\mp$ (or $\mm$) is an unstable fixed point of $\ft$.

\end{thm}
\begin{proof}
This uses the same method as the proofs of results in section
\ref{momsec}.

First, it is clear that $S^+$ is a solution to (\ref{RDE3}), where
$P(S^+=1)=\mp=1-P(S^+=0)$. Now take interleaved tree-indexed solutions to
the RDE on the tree $\tT$, corresponding (on consecutive layers) to mean $\mp$ and $\mm$.
Then we define
$C_{(n)}^+=P(S^+_\emptyset=1|N_v; |v|\leq 2n)= 1 - \prod_{i_1=1}^{N_\emptyset}(1 -
\prod_{{i_2}=1}^{N_{i_1}}(...(1-(\mp)^{N_{i_1i_2\ldots i_{2n-
1}}})...)) $. It follows 
that $C_{(n)}^+$ converges a.s. and in $L^2$ to $C^+$ and that this must
be the unique endogeneous solution (since if $Z$ is any solution with
mean $\mp$ then $E[Z_\emptyset|N_v; |v|\leq 2n]=C_{(n)}^+$).

As in Lemma \ref{momlem}, we establish that there are at most two solutions
by showing that there are at most two possible moment sequences for a
solution and that if $\mp$ is stable (for $\ft$) then the only possible
moment sequence corresponds to the discrete solution $S^+$.

To do this, note that, denoting a possible moment sequence starting with first
moment $\mp$ by $(\mu^k_+)$, we have
\begin{equation}\label{R4}
H(\mmk)=H(\sum_{j=0}^k(-1)^j{{k}\choose{j}}H(\mpj))=\sum_{j=0}^k(-
1)^j{{k}\choose{j}}\mpj.
\end{equation}

Then we look for
solutions of
\begin{equation}\label{R5}
H(\sum_{j=0}^{k-1}(-1)^j{{k}\choose{j}}H(\mpj)+(-1)^kH(t))=\sum_{j=0}^{k-1}(-
1)^j{{k}\choose{j}}\mpj+(-1)^kt,
\end{equation}
in the range where the argument of $H$ on the lefthand side is non-
negative and less than 1. In this range $H$ is increasing and convex so
there are at most two solutions.

Suppose that $\mp$ is a stable fixed point then the
unique moment sequence is constant, since the other solution of
$$
g(t)\defto H(1-2H(\mp)+H(t))-(1-2\mp+t)=0
$$
must be greater than $\mp$ (because $g'(\mp)=H'(\mp)H'(\mm)-1\leq 0$).

If $\mp$ is unstable, then there are, potentially two solutions for
$\mu^2_+$, one of which is $\mp$. Taking the other potential solution,
and seeking to solve (\ref{R5}),
one of the solutions will give a value
for $\mmk$ greater than $\mu^*>\mu^2_-$ which is not feasible, so there
will be at most one sequence with $\mu^2_+\neq\mp$.

Now, as in the proof of Theorem 4.9, we can show that, if $\mp$ is unstable then, in the
corresponding RDE with branching factor truncated by $n$, the two
solutions to the RDE are distinct for large $n$, and the
endogenous solution converges to $C^+$ in $L^2$ as $n\rightarrow \infty$. It follows that there
are two distinct solutions in this case.

\end{proof}

Given a fixed point $\mp$ of $\ft$, denote the law of the corresponding
conditional probability solution by $\varsigma_+$. Denote the
corresponding basin of attraction (under $\T^2$) by $B(\vp)$ and denote
the basin of attraction of $\mp$ under the map $\ft$ by $b^2(\mp)$. Then

\begin{thm}
the following dichotomy holds:
\begin{itemize}
\item[(i)]if $H'(\mp)H'(\mm)>1$, then
$$
B(\vp)=\{\pi:\, \pi\hbox{ has mean }\mp\hbox{ and }\pi\hbox{ is not
concentrated on }\{0,1\}\}.
$$
\item[(ii)]if $H'(\mp)H'(\mm)\leq 1$ then
$$
B(\vp)=\{\pi:\, \pi\hbox{ has mean }m\in b^2(\mp) \}
$$

\end{itemize}

\end{thm}
\begin{proof}
This can be proved in exactly the same way as Theorems 5.3 and 5.4.
\end{proof}

\section{Examples}
We conclude with some examples.

\begin{example}
We consider first the case where $N$ is Geometric($\alpha$), so that
$P(N=k)=\beta^{k-1}\alpha$ and $H(s)=\frac{\alpha s}{1-\beta s}$ (with $\beta=1-\alpha$). It follows that
$$
\ft(s)=s,
$$
so that every pair  $(s,\frac{1-s}{1-qs})$ is a two-cycle of $f$ and
the unique fixed point of $f$ is $1-\sqrt \alpha$. It also follows that $s$
is a neutrally stable fixed point of $\ft$ for each $s\in[0,1]$. Thus we
see that the unique endogenous solution to the original RDE is discrete
and the value at the root of the tree is the a.s. limit of $1-\prod_{i_1=1}^{N_\emptyset}(1-
\prod_{i_2=1}^{N_{i_1}}(\ldots (1-(1-\sqrt
\alpha)^{N_{i_1,\ldots,i_n}})\ldots ))$. Moreover, for any $s$, there is a
unique solution to the iterated RDE with mean $s$ and it is discrete and
endogenous and is the a.s. limit of $1-\prod_{i_1=1}^{N_\emptyset}(1-
\prod_{i_2=1}^{N_{i_1}}(\ldots (1-s^{N_{i_1,\ldots,i_{2n-1}}})\ldots )).$

\end{example}

\begin{example}
Consider the original noisy veto-voter model on the binary tree. It
follows from (\ref{trans}) that
$$
H(z)=(pH(z)+qz)^2\Rightarrow H(z)=\frac{1-2pqz-\sqrt{1-4pqz}}{2p^2}.
$$
This is non-defective if and only if $p\leq \frac{1}{2}$ (naturally),
i.e. if and only if extinction is certain in the trimmed tree from the
original  veto-voter model.
It is fairly straightforward to show that $H'(\mu^1)>1\Leftrightarrow
p<\frac{1}{2}$. Thus, the endogenous solution is non-discrete precisely
when the trimmed tree is sub-critical.

\end{example}

\begin{example}
In contrast to the case of the veto-voter model on the binary tree, the veto-voter model on a trinary tree can
show a non-endogenous discrete solution even when the trimmed tree is
supercritical.
More precisely, the trimmed tree is supercritical precisely when
$p>\frac{1}{3}$, but  the
discrete solution is non-endogenous if and only if
$p<p_e^{(3)}\defto\frac{3.\sqrt 3-4}{3.\sqrt 3-2}$, and
$p_e^{(3)}>\frac{1}{3}$.
\end{example}

\end{document}